\documentclass{amsart}

\title[A new kind of anti-foundation axioms]%
      {A new kind of anti-foundation axioms}

\author[Ju]{Daheng Ju}
 \address[Daheng Ju]
         {Department of Philosophy and Religious Studies, Peking University, 5 Yiheyuan Road, Beijing, 100871 China}
 \email{judaheng@pku.edu.cn}

\usepackage{latexsym,amsfonts,amsmath,amssymb,mathrsfs,tikz-cd,adjustbox}
\usetikzlibrary{fit,decorations.pathmorphing,backgrounds}
\usepackage{changepage,calc}
\usepackage{csquotes}
\usepackage{braket}
\usepackage{bbm}
\usepackage[hidelinks]{hyperref}
\usepackage{relsize}
\usepackage[cmyk,svgnames,dvipsnames]{xcolor}
\usepackage{tikz}
\usetikzlibrary {3d}
\usepackage{comment}
\usetikzlibrary{arrows,arrows.meta,petri,topaths,positioning,shapes,shapes.misc,patterns,calc,decorations.pathreplacing,hobby}
\pgfdeclarelayer{foreground}
\pgfdeclarelayer{background}
\pgfdeclarelayer{boardarrows}
\pgfdeclarelayer{boardgrid}
\pgfdeclarelayer{boardshades}
\pgfsetlayers{boardshades,boardgrid,boardarrows,background,main,foreground}
\usepackage{wrapfig} 
\usepackage{float}
\usepackage[utf8]{inputenc}
\RequirePackage{doi}

\usepackage{enumitem}

\usepackage[backend=bibtex,style=alphabetic,maxbibnames=15,maxcitenames=6,dateabbrev=false]{biblatex}
\renewcommand{\UrlFont}{} 
\renewbibmacro{in:}{\ifentrytype{article}{}{\printtext{\bibstring{in}\intitlepunct}}} 
\DeclareFieldFormat{url}{\UrlFont\url{#1}} 
\DeclareFieldFormat{urldate}{
  (version \thefield{urlday}\addspace%
  \mkbibmonth{\thefield{urlmonth}}\addspace%
  \thefield{urlyear}\isdot)}
\DeclareFieldFormat{eprint:arxiv}{
  \ifhyperref
    {\href{http://arxiv.org/abs/#1}{%
        arXiv\addcolon\nolinkurl{#1}}\iffieldundef{eprintclass}{}{\UrlFont{\mkbibbrackets{\thefield{eprintclass}}}}}
    {arXiv\addcolon\nolinkurl{#1}\iffieldundef{eprintclass}{}{\UrlFont{\mkbibbrackets{\thefield{eprintclass}}}}}}

\newcommand{\ZFC}{\mathsf{ZFC}}
\newcommand{\zfc}{\mathsf{ZFC}}
\newcommand{\AF}{\mathsf{AF}}
\newcommand{\BAFA}{\mathsf{BAFA}}
\newcommand{\FAFA}{\mathsf{FAFA}}
\newcommand{\SAFA}{\mathsf{SAFA}}
\newcommand{\AFA}{\mathsf{AFA}}
\newcommand{\afa}{\mathsf{AFA}}
\newcommand{\trcl}{\text{trcl}}
\DeclareMathOperator{\uni}{\mathbf{V}}
\DeclareMathOperator{\wf}{\mathbf{WF}}
\DeclareMathOperator{\con}{Con}

\newtheorem{theorem}{Theorem}
\newtheorem{lemma}{Lemma}
\newtheorem{definition}{Definition}
\newtheorem{fact}{Fact}
\newtheorem{corollary}{Corollary}
\newtheorem{question}{Question}

\begin{document}

\begin{abstract}

Among the four well-known anti-foundation axioms, $\BAFA$, $\FAFA$, $\SAFA$, and $\AFA$, the latter three are all special cases of $\AFA^\sim$ ($\sim$ is one of the regular bisimulations). In this paper, we generalize it to $\AFA^\sim$ ($\sim$ is regular and has the so-called set property), and show that this generalization is substantial by constructing new anti-foundation axioms.

\end{abstract}

\maketitle

\tableofcontents

\newcommand\JD[1]{{\color{blue}(JD: #1)}}

\section{Introduction}

In non-well-founded set theory, we replace the axiom of foundation $\AF$ in $\ZFC$ with one of the anti-foundation axioms asserting the existence of non-well-founded sets. There are four well-known anti-foundation axioms: $\BAFA$, $\FAFA$, $\SAFA$, and $\AFA$, where the latter three are all special cases of $\AFA^\sim$ ($\sim$ is one of the regular bisimulations). For this reason, anti-foundation axioms are often identified with $\BAFA$ and $\AFA^\sim$ ($\sim$ is one of the regular bisimulations), and the study of anti-foundation axioms is often identified with the study of them (see, for example, \cite{Daghighi}\cite{Ju}).

In this paper, we will demonstrate the limitations of using regular bisimulations to induce anti-foundation axioms. In particular, we consider the following rigorous question:

\begin{question}[finitary arithmetic]
If $\zfc^-\vdash$``$\sim$ is a regular graph relation'' and $\zfc^-+\afa^\sim$ is consistent relative to $\zfc^-$, does there necessarily exist some $\sim'$ such that $\zfc^-\vdash$``$\sim'$ is a regular bisimulation and $\afa^\sim\Leftrightarrow\afa^{\sim'}$''?
\end{question}

By constructing new anti-foundation axioms, we will give a negative answer to this question.

Unless otherwise specified, throughout this paper, we work in $\ZFC-\AF$, abbreviated as $\ZFC^-$.

\section{A brief introduction to anti-foundation axioms}

In this section, we give a brief introduction to $\BAFA$ and $\AFA^\sim$ ($\sim$ is one of the regular bisimulations).

The standard form of an anti-foundation axiom is ``a graph is an exact picture iff such-and-such''. In the context of non-well-founded set theory, ``graph'' means ``accessible pointed directed graph'', that is, a directed graph $G$ together with a distinguished vertex called the point (i.e. a structure $(V_G$(also denoted by $G$)$,E_G,p_G)$) where for every $a \in G$, there exists a (finite) path from $p_G$ to $a$, and $G$ is called an exact picture (of some set $X$) if it is isomorphic with $G_X = (\trcl(\{X\}), \in, X)$, called the canonical picture of $X$.

As Incurvati says:
\begin{quote}
Since every set has, up to isomorphism, one and only one exact picture, we can know which sets there are if we know which apgs are exact pictures. Each non-well-founded set theory gives a different answer to this question, and the role of the anti-foundation axiom can be seen as precisely that of specifying which apgs are exact pictures.\cite{Incurvati}
\end{quote}

We denote the set of children of $a$ by $C_G(a)$. $G$ is called extensional if for every $a_1,a_2 \in G$, $C_G(a_1) = C_G(a_2)$ implies $a_1 = a_2$. By the axiom of extensionality, exact pictures are extensional. $\BAFA$ is equivalent with the conjunction of two propositions $\mathsf{BA}_1$ and $\mathsf{BA}_2$, where $\mathsf{BA}_1$ is the proposition ``a graph is an exact picture iff it is extensional''.

Given $a \in G$, we denote by $G[a]$ the subgraph below $a$, that is, the subgraph induced by the set $\{ b \in G \mid \text{there exists a path from } a \text{ to } b \}$ having $a$ as point. Such subgraphs are called descendant subgraphs. Given a binary relation between graphs (a graph relation for short) $\sim$, which is an $\mathcal{L}_\in$-formula $\sim(x,y)$, $G$ is called $\sim$-extensional if for every $a_1,a_2 \in G$, $G[a_1] \sim G[a_2]$ implies $a_1 = a_2$.\footnote{It is worth noticing that, write $G\approx G'$ if $\{G[a]\mid a\in C_G(p_G)\}=\{G'[a']\mid a'\in C_{G'}(p_{G'})\}$, then being extensional is equivalent with being $\approx$-extensional.} An anti-foundation axiom is called normal in \cite{Ju} if it demands that ``every exact pictures is $\cong$-extensional'', which is equivalent with ``every exact picture is an exact picture of one and only one set''.

The simplest normal anti-foundation axiom is $\FAFA$: a graph is an exact picture iff it is extensional and $\cong$-extensional. Next, we introduce regular bisimulations, which can systematically induce normal (and relatively consistent) anti-foundation axioms, including $\FAFA$, $\SAFA$, and $\AFA$.

\begin{definition}
    A graph relation $\sim$ is called regular if the following hold:
\begin{enumerate}[label=(\roman*)]
    \item\ $\sim$ is an equivalence relation.
    \item\ $G \approx G'$ implies $G \sim G'$.
    \item\ $G \cong G'$ implies $G \sim G'$.
\end{enumerate}
\end{definition}

Every regular graph relation $\sim$ induces a normal anti-foundation axiom $\AFA^\sim$: a graph is an exact picture iff it is $\sim$-extensional. $\AFA^\sim$ is the conjunction of two propositions:

\begin{itemize}
    \item ($\AFA_1^\sim$) A graph is an exact picture if it is $\sim$-extensional.
    \item ($\AFA_2^\sim$) A graph is an exact picture only if it is $\sim$-extensional.
\end{itemize}

\begin{definition}
    A graph relation $\sim$ is called a bisimulation if $G \sim G'$ implies the following:

\begin{enumerate}[label=(\roman*)]
    \item\ For every $a \in C_G(p_G)$, there exists some $a' \in C_{G'}(p_{G'})$ such that $G[a] \sim G'[a']$.
    \item\ For every $b' \in C_{G'}(p_{G'})$, there exists some $b \in C_G(p_G)$ such that $G[b] \sim G'[b']$.
\end{enumerate}
\end{definition}

Aczel noted that $\FAFA$, $\SAFA$, and $\AFA$ can all be induced by one of the regular bisimulations.\cite{Aczel}

First, consider $\FAFA$. Given a graph $G$, we define $G^*$ as follows: if $p_G$ has no parents, then $G^* = G$; if $p_G$ has a parent, then $G^*$ is a new graph formed by adding to $G$ a new vertex $p_{G^*}$ and new edges $(p_{G^*}, a)$ for every $a \in C_G(p_G)$ having $p_{G^*}$ as point. We write $G \cong^* G'$ if $G^* \cong {G'}^*$. It's easy to check that:

\begin{fact} \mbox{}

\begin{enumerate}[label=(\arabic*)]
    \item\ Being $\cong^*$-extensional is equivalent with being both extensional and $\cong$-extensional. As a corollary, $\FAFA$ is equivalent with $\AFA^{\cong^*}$: a graph is an exact picture iff it is $\cong^*$-extensional.
    \item\ $\cong^*$ is the transitive closure of $\approx \cup \cong$ (thus, it is the minimum regular graph relation: $\cong^*$ is a regular graph relation, and if $\sim$ is a regular graph relation, then $G\cong^* G'\Rightarrow G\sim G'$).
    \item\ $\cong^*$ is a bisimulation (thus, it is the minimum regular bisimulation).
\end{enumerate}
\end{fact}

Second, consider $\SAFA$. We denote by $G^t$ the unfolding of $G$ (for the precise definition, see \cite{Aczel}). We write $G \cong^t G'$ if $G^t \cong {G'}^t$. $\cong^t$ is a regular bisimulation. $\SAFA$ is just $\AFA^{\cong^t}$: a graph is an exact picture iff it is $\cong^t$-extensional.

Finally, consider $\AFA$. 
We define the maximum bisimulation $\equiv_\mathrm{M}$ as follows: let $G[G]=\{G[a]\mid a\in G\}$, $G\equiv_\mathrm{M} G'$ iff there exists some $B\subseteq (G[G]\cup G'[G'])\times(G[G]\cup G'[G'])$ such that $B$ is a bisimulation on $G[G]\cup G'[G']$ and $(G,G')\in B$. It's easy to check that $\equiv_\mathrm{M}$ is regular (thus, it is also the maximum regular bisimulation). $\AFA$ is just $\AFA^{\equiv_\text{M}}$: a graph is an exact picture iff it is $\equiv_\text{M}$-extensional.

Aczel proved that if $\sim$ is a regular bisimulation and has a certain absoluteness property, then $\AFA^\sim$ is consistent with $\ZFC^{-}$ (relative to $\ZFC^{-}$).\cite{Aczel} Below we present a proof that we have improved.

\begin{lemma}
    If $\sim$ is a regular bisimulation, $G,G'$ are $\sim$-extensional, then $G\sim G'\Rightarrow G\cong G'$.
\end{lemma}
\begin{proof}
    If $G,G'$ are $\sim$-extensional, suppose that $G\sim G'$.
    Since $\sim$ is a bisimulation, for every $a\in G$, there exists some $a'\in G'$ such that $G[a]\sim G'[a']$, and since $G'$ is $\sim$-extensional, such $a'$ is unique.
    This gives us a map $f$ from $G$ to $G'$.
    It's easy to check that $f$ is a bijection, $a\to b\Leftrightarrow f(a)\to f(b)$, and $f(p_G)=p_{G'}$, i.e., $f$ is an isomorphism from $G$ to $G'$.
\end{proof}

\begin{lemma}
    $G$ is called quasi-$\sim$-extensional if for every $a,b\in G\backslash\{p_G\}$, $G[a]\sim G[b]\Rightarrow a=b$.
    If $\sim$ is a regular bisimulation, $G$ has a point having no parents and is quasi-$\sim$-extensional, then for every $a\in G\backslash \{p_G\}$, $G\sim G[a]\Rightarrow G\approx G[a]$.
\end{lemma}
\begin{proof}
    If $G$ has a point having no parents and is quasi-$\sim$-extensional, suppose that $G\sim G[a]$ ($a\in G\backslash \{p_G\}$).
    Since $\sim$ is a bisimulation, for every $b\in C_G(p_G)$, there exists some $b'\in C_G(a)$ such that $G[b]\sim G[b']$, and since $b,b'\in G\backslash \{p_G\}$, $b=b'$, which implies that $C_G(p_G)\subseteq C_G(a)$.
    Similarly, $C_G(a)\subseteq C_G(p_G)$, thus we have proved that $G\approx G[a]$.
\end{proof}

\begin{definition}
    Given a class model $\mathbf{M}$ and $x\in\mathbf{M}$ such that $\mathbf{M}\vDash$``$x$ is a graph'' (thus, $x=(V,E,p)^\mathbf{M}$), we define $ext_\mathbf{M}(x)=(\{a\mid a\in_\mathbf{M}V\},\{(a,b)\mid (a,b)^{\mathbf{M}}\in_{\mathbf{M}}E\},p)$.
    $\sim$ is called absolute for $\mathbf{M}$ if for every $x,y\in \mathbf{M}$, $(x\sim y)^\mathbf{M}\Leftrightarrow ext_\mathbf{M}(x)\sim ext_\mathbf{M}(y)$.
\end{definition}

Definition 3 is a generalization of the well-known ``absoluteness for some transitive class model'': if $\mathbf{M}$ is a transitive class model, then $ext_\mathbf{M}(x)=x$, so the definition degenerates to: for every $x,y\in \mathbf{M}$, $(x\sim y)^\mathbf{M}\Leftrightarrow x\sim y$.

\begin{theorem}\cite{Aczel}
    If $\sim$ is a regular bisimulation, let $\uni_0^\sim$ be the class of $\sim$-extensional graphs, and $\uni_c^\sim$ be the class of sets of minimal-rank representatives of $\uni_0^\sim$ modulo $\sim$ (equipped with a natural graph structure $\leftarrow$, it is an $\mathcal{L}_\in$-structure $(\uni_c^\sim,\leftarrow)$), then $\uni_c^\sim$ is a class model of $\zfc^-$; if in addition $\sim$ is absolute for $\uni_c^\sim$, then $\uni_c^\sim$ is a class model of $\zfc^-+\afa^\sim$.
\end{theorem}

By lemma 1, $\uni_c^\sim$ can be equivalently defined as the class of sets of minimal-rank representatives of $\uni_0^\sim$ modulo $\cong$.

Denote the quotient map from $\uni_0^\sim$ to $\uni_c^\sim$ by $\pi$, given a $\sim$-extensional graph $G$, $\pi$ induces an isomorphism $\pi_G$ from $G$ to $\uni_c^\sim[\pi(G)]$: $\pi_G(a)=\pi(G[a])$.
Now, $\uni_c^\sim[\pi(G)]$ is a representative of the equivalence class of $G$ which is independent of $G$. Thus, $\uni_c^\sim$ can be further equivalently defined as the class of representatives of $\uni_0^\sim$ modulo $\cong$.
\begin{proof}
    It's sufficient to prove that: if $\sim$ is a regular bisimulation, then the following hold:
    \begin{enumerate}[label=(\roman*)]
        \item Every descendant subgraph of $\uni_c^\sim$ is $\sim$-extensional.
        \item Every $\sim$-extensional graph $G$ is isomorphic with some descendant subgraph of $\uni_c^\sim$.
        \item For every $X\subseteq \uni_c^\sim$, there exists some $x\in \uni_c^\sim$ such that $C_{\uni_c^\sim}(x)=X$.
    \end{enumerate}
    
    This is because: if (i) (iii) hold, by Rieger's theorem\cite{Rieger}, $\uni_c^\sim$ is a class model of $\zfc^-$; if in addition (i) (ii) hold and $\sim$ is absolute for $\uni_c^\sim$, then $\uni_c^\sim$ is a class model of $\afa^\sim$.
    
    For (i), every descendant subgraph of $\uni_c^\sim$ is of the form $\uni_c^\sim[\pi(G)]$ ($G$ is $\sim$-extensional), so $\uni_c^\sim[\pi(G)] \cong G$ is $\sim$-extensional.
        
    For (ii), obviously $\uni_c^\sim[\pi(G)] \cong G$ is what's required.
    
    For (iii), let $G_0$ be the union of all $\uni_c^\sim[x]$ ($x \in X$), and $G_1$ be the graph formed by adding to $G_0$ a new vertex $p_{G_1}$ and new edges $(p_{G_1}, x)$ for every $x \in X$ having $p_{G_1}$ as point. By lemma 1, $G_1$ is quasi-$\sim$-extensional. If $G_1$ is $\sim$-extensional, then $\pi(G_1)$ is what's required. If not, there exists some $a\in G_0$ such that $G_1\sim G_1[a]$. By lemma 2, $G_1\approx G_1[a]$, so $a$ is what's required.
\end{proof}

\begin{corollary}
    (finitary arithmetic) If $\zfc^-$+``$\sim$ is a regular bisimulation and is absolute for $\uni_c^\sim$'' is consistent, then $\zfc^-+\afa^\sim$ is consistent; in particular, if $\zfc^-\vdash$``$\sim$ is a regular bisimulation and is absolute for $\uni_c^\sim$'', then $\zfc^-+\afa^\sim$ is consistent relative to $\zfc^-$.
\end{corollary}

\section{The generalization of Aczel's theorem}

In this section, we generalize Aczel's theorem proved in the previous section and give the ``necessary and sufficient condition'' for a regular graph relation $\sim$ to induce a relatively consistent anti-foundation axiom $\afa^\sim$.

\begin{definition}
    $\sim$ is called having the set property if for every set $X$ of mutually non-isomorphic $\sim$-extensional graphs, there exists some $\sim$-extensional graph $G$ such that $\{G[a]\mid a\in C_G(p_G)\}=X$ up to isomorphism.
\end{definition}

\begin{theorem}
    If $\sim$ is regular and $\afa^\sim$ holds, then $\sim$ has the set property.
\end{theorem}
\begin{proof}
    Under $\afa^\sim$, suppose that $X$ is a set of mutually non-isomorphic $\sim$-extensional graphs. For every $G \in X$, there exists one (and only one) $y$ such that $G$ is an exact picture of $y$. Consider the set $Y$ consisting of these $y$, $G_Y$ is what's required.
\end{proof}

\begin{theorem}
    If $\sim$ is regular, then the following are mutually equivalent:
    \begin{enumerate}[label=(\Roman*)]
        \item $\sim$ has the set property.
        \item If $G$ has a point having no parents and is quasi-$\cong$-extensional, and for every $a\in G\backslash\{p_G\}$, $G[a]$ is $\sim$-extensional, then either $G$ is $\sim$-extensional, or there exists some $a\in G\backslash\{p_G\}$ such that $G\approx G[a]$.
        \item the following hold:
            \begin{enumerate}[label=(\roman*)]
                \item If $G\not\cong G'$ are both $\sim$-extensional, then $G\not\sim G'$.
                \item If $G$ has a point having no parents and is quasi-$\sim$-extensional, then for every $a\in G\backslash\{p_G\}$, $G\sim G[a]\Rightarrow G\approx G[a]$.
            \end{enumerate}
    \end{enumerate}
\end{theorem}
\begin{proof}
    First, we prove that (I) implies (III)(i).

    Since $\sim$ has the set property, there exists some $\sim$-extensional graph $G''$ such that $\{G''[a]\mid a\in C_{G''}(p_{G''})\}=\{G,G'\}$ up to isomorphism, so $G\not\sim G'$.

    Second, we prove that (I) implies (III)(ii).

    Let $\{G[b]\mid b\in C_G(p_G)\}=X$, $X$ is a set of mutually non-isomorphic $\sim$-extensional graphs. Since $\sim$ has the set property, there exists some $\sim$-extensional graph $G'$ such that $G'\approx G$ up to isomorphism.
    
    If $G'$ has a point having no parents, then $G'\cong G$. Thus, $G$ is $\sim$-extensional, and there doesn't exist any $a\in G\backslash\{p_G\}$ such that $G\sim G[a]$.
    
    If $G'$ has a point having a parent, then $G'\cong$ some descendant subgraph of some $G[b]$ ($b\in C_G(p_G)$), i.e., $G'\cong$ some $G[b]$ ($b\in G\backslash\{p_G\}$), so $G[b]\approx G$. Now, if $G\sim G[a]$ ($a\in G\backslash\{p_G\}$), then $G[a]\sim G[b]$, so $a=b$, and then $G[a]=G[b]\approx G$.

    Third, we prove that (III) implies (II).

    If such a $G$ is not $\sim$-extensional, by (i), it is quasi-$\sim$-extensional, so $G\sim G[a]$ for some $a\in G\backslash\{p_G\}$, and by (ii), $G\approx G[a]$.

    Last, we prove that (II) implies (I).

    For every set $X$ of mutually non-isomorphic $\sim$-extensional graphs, we may assume without loss of generality that any two graphs among their descendant subgraphs that are isomorphic are identical.
    Let $G_0$ be the union of all $G \in X$, and $G_1$ be the graph formed by adding to $G_0$ a new vertex $p_{G_1}$ and new edges $(p_{G_1}, p_G)$ for every $G \in X$ having $p_{G_1}$ as point.
    By (II), $G_1\approx $ some $\sim$-extensional graph, which is what's required.
\end{proof}

\begin{corollary}
    If $\sim$ is a regular bisimulation, then $\sim$ has the set property.
\end{corollary}
\begin{proof}
    It suffices to prove that $\sim$ satisfies (III) in Theorem 3, and this is precisely Lemma 1 and Lemma 2.
\end{proof}

\begin{theorem}
    If $\sim$ is regular and has the set property, let $\uni_0^\sim$ be the class of $\sim$-extensional graphs, and $\uni_c^\sim$ be the class of sets of minimal-rank representatives of $\uni_0^\sim$ modulo $\sim$ (equipped with a natural graph structure $\leftarrow$, it is an $\mathcal{L}_\in$-structure $(\uni_c^\sim,\leftarrow)$), then $\uni_c^\sim$ is a class model of $\zfc^-$; if in addition $\sim$ is absolute for $\uni_c^\sim$, then $\uni_c^\sim$ is a class model of $\zfc^-+\afa^\sim$.
\end{theorem}

As in the case of Theorem 1, $\uni_c^\sim$ can be equivalently defined as the class of representatives of $\uni_0^\sim$ modulo $\cong$.

\begin{proof}
    Note that the proof of Theorem 1 uses only the fact that a regular bisimulation satisfies (III) in Theorem 3, which is equivalent with the set property. Thus ends the proof.
\end{proof}

\begin{corollary}
    (finitary arithmetic) If $\zfc^-\vdash$``$\sim$ is regular'', then:
    \begin{enumerate}[label=(\arabic*)]
        \item If $\zfc^-\vdash$``$\sim$ doesn't have the set property'', then $\zfc^-+\afa^\sim$ is inconsistent.
        \item If $\zfc^-+$``$\sim$ has the set property and is absolute for $\uni_c^\sim$'' is consistent, then $\zfc^-+\afa^\sim$ is consistent; in particular, if $\zfc^-\vdash$``$\sim$ has the set property and is absolute for $\uni_c^\sim$'', then $\zfc^-+\afa^\sim$ is consistent relative to $\zfc^-$.
    \end{enumerate}
\end{corollary}

\section{A new anti-foundation axiom family}

In this section, we construct a new anti-foundation axiom family based on Corollary 3 and use it to give a negative answer to Question 1.

\begin{theorem}
    (finitary arithmetic) There exists a graph relation $\sim$ such that $\zfc^-\vdash$``$\sim$ is regular'', $\zfc^-+\afa^\sim$ is consistent relative to $\zfc^-$, and for every $\sim'$ such that $\zfc^-\vdash$``$\sim'$ is a regular bisimulation'', (if $\zfc^-$ is consistent, then) $\zfc^-\not\vdash\afa^\sim\Leftrightarrow\afa^{\sim'}$.
\end{theorem}

\begin{proof}
    Denote by $Q_n$ the graph consisting of $n$ ($\geq 1$)-many vertices with self-loops $a_1\to\dots\to a_n$ having $a_1$ as point.
    
\begin{figure}[htbp]
\[\begin{tikzcd}
	{a_1} \\
	\\
	{a_2} \\
	\\
	{a_3}
	\arrow[from=1-1, to=1-1, loop, in=325, out=35, distance=10mm]
	\arrow[from=1-1, to=3-1]
	\arrow[from=3-1, to=3-1, loop, in=325, out=35, distance=10mm]
	\arrow[from=3-1, to=5-1]
	\arrow[from=5-1, to=5-1, loop, in=325, out=35, distance=10mm]
\end{tikzcd}\]
\caption{$Q_3$}
\end{figure}

    For $n\geq2$, define $\cong_{Q_n}^*$ as follows: $G\cong_{Q_n}^*G'$ iff $G\cong^*G'$ or $G,G'\in\{Q_1,Q_1^*,Q_n,Q_n^*\}$ up to isomorphism.
    
    Obviously, $\cong_{Q_2}^*$ is a regular bisimulation. For $n\geq3$, $\cong_{Q_n}^*$ is regular, and below we prove that $\cong_{Q_n}^*$ satisfies the other requirements.

    First, we prove that $\zfc^-+\afa^{\cong_{Q_n}^*}$ is consistent relative to $\zfc^-$.
    
    By Corollary 3, it's sufficient to prove under $\zfc^-$ that $\cong_{Q_n}^*$ has the set property and is absolute for $\uni_c^{\cong_{Q_n}^*}$.
    The latter is obvious. Below we prove the former, and it's sufficient to prove that $\cong_{Q_n}^*$ satisfies (III) in Theorem 3.
    
    (i): If $G\not\cong G'$ are both $\cong_{Q_n}^*$-extensional, then they are both $\cong^*$-extensional, so $G\not\cong^*G'$.
    Assume the contrary that $G\cong_{Q_n}^* G'$, comparing $\cong^*$ and $\cong_{Q_n}^*$, we have that $G,G'\in\{Q_1,Q_1^*,Q_n,Q_n^*\}$ up to isomorphism. Since $G,G'$ are both $\cong_{Q_n}^*$-extensional, and among $Q_1,Q_1^*,Q_n,Q_n^*$, only $Q_1$ is $\cong_{Q_n}^*$-extensional, we have that $G,G'\cong Q_1$, contrary to $G\not\cong G'$!
    Thus, $G\not\cong_{Q_n}^* G'$.
    
    (ii): If $G$ has a point having no parents and is quasi-$\cong_{Q_n}^*$-extensional, then it is quasi-$\cong^*$-extensional, so for every $a\in G\backslash\{p_G\}$, $G\cong^* G[a]\Rightarrow G\approx G[a]$.
    Assume the contrary that there exists some $a\in G\backslash\{p_G\}$ such that $G\cong_{Q_n}^* G[a]$ and $G\not\approx G[a]$, comparing $\cong^*$ and $\cong_{Q_n}^*$, we have that $G,G[a]\in\{Q_1,Q_1^*,Q_n,Q_n^*\}$ up to isomorphism. Since $G$ has a point having no parents and is quasi-$\cong_{Q_n}^*$-extensional, we have that $G\cong Q_1^*$, so $G[a]\cong Q_1$, contrary to $G\not\approx G[a]$!
    Thus, for every $a\in G\backslash\{p_G\}$, $G\cong_{Q_n}^* G[a]\Rightarrow G\approx G[a]$.
        
    Second, we prove that for every $\sim$ such that $\zfc^-\vdash$``$\sim$ is a regular bisimulation'', (if $\zfc^-$ is consistent, then) $\zfc^-\not\vdash\afa^{\cong_{Q_n}^*}\Leftrightarrow\afa^{\sim}$.
    
    Working in $\zfc^-$, we prove that $\afa^{\cong_{Q_n}^*}\Rightarrow\neg\afa^\sim$. Note that $Q_2$ is $\cong_{Q_n}^*$-extensional and $Q_n$ is not $\cong_{Q_n}^*$-extensional.
    Assume the contrary that $\afa^{\cong_{Q_n}^*}\land\afa^\sim$, then $Q_2$ is $\sim$-extensional and $Q_n$ is not $\sim$-extensional, i.e., there exist some $1\leq i<j\leq n$ such that $Q_i\sim Q_j$. By repeatedly applying the condition that $\sim$ is a regular bisimulation, we have that $Q_1\sim Q_2$, so $Q_2$ is not $\sim$-extensional, a contradiction! Thus, $\afa^{\cong_{Q_n}^*}\Rightarrow\neg\afa^\sim$.
    
    Back to finitary arithmetic, we have that $\zfc^-+\afa^{\cong_{Q_n}^*}\vdash\neg\afa^\sim$, so $\con(\zfc^-)$ implies $\con(\zfc^-+\afa^{\cong_{Q_n}^*})$, which in turn implies $\con(\zfc^-+\afa^{\cong_{Q_n}^*}\not\Leftrightarrow\afa^{\sim})$.
\end{proof}

\begin{corollary}
    The converse of Corollary 2 does not hold.
\end{corollary}

Theorem 5 gives a new anti-foundation axiom family $\afa^{\cong_{Q_n}^*}$ ($n\geq3$): a graph is an exact picture iff it is $\cong_{Q_n}^*$-extensional. It is easy to prove that being $\cong_{Q_n}^*$-extensional is equivalent with being $\cong^*$-extensional and having no $Q_n$ as a descendant subgraph. Thus, $\afa^{\cong_{Q_n}^*}$ is equivalent with: a graph is an exact picture iff it is $\cong^*$-extensional and has no $Q_n$ as a descendant subgraph.

The family $\afa^{\cong_{Q_n}^*}$ ($n\geq3$) witnesses that the answer to Question 1 is negative (unless $\ZFC^-$ itself is inconsistent). Thus, $\AFA^\sim$ ($\sim$ is one of the regular bisimulations) doesn't exhaust all normal (and relatively consistent) anti-foundation axioms, and should be generalized at least to $\AFA^\sim$ ($\sim$ is regular and has the set property).

\section{Another new anti-foundation axiom family}

When using regular bisimulations to induce anti-foundation axioms, the most important ones are the minimum regular bisimulation $\cong^*$ and the maximum regular bisimulation $\equiv_\mathrm{M}$. Now, we use regular graph relations having the set property to induce anti-foundation axioms, so a natural question is:

\begin{question}
    Does there exist the maximum regular graph relation having the set property?
\end{question}

(The minimum regular graph relation having the set property is $\cong^*$ as well.)

In this section, we will construct another new anti-foundation axiom family and use it to give a negative answer to this question.

Recall that a regular graph relation $\sim$ has the set property iff the following hold:

\begin{enumerate}[label=(\roman*)]
    \item If $G\not\cong G'$ and $G\sim G'$, then at least one of them is not $\sim$-extensional. Equivalently, if non-*-graphs $G\not\cong G'$ and $G\sim G'$, then at least one of them is not $\sim$-extensional (a graph $G$ is called a *-graph if $G=(G_0)^*$, where $G_0$ has a point having a parent).
    \item If $G$ has a point having no parents, $a\in G\backslash\{p_G\}$ such that $G\not\approx G[a]$ and $G\sim G[a]$, then there exist some $b \not= c \in G\backslash\{p_G\}$ such that $G[b] \sim G[c]$.
\end{enumerate}

A natural idea is to define the maximum regular graph relation having the set property by modifying the definition of the maximum bisimulation:

\begin{definition}
    $G\sim_\mathrm{M} G'$ iff there exists some $B\subseteq (G[G]\cup G'[G'])\times(G[G]\cup G'[G'])$ such that $B$ is a regular graph relation satisfying (i)(ii) on $G[G]\cup G'[G']$ and $(G,G')\in B$.
\end{definition}

Admittedly, in order to prove that $\sim_\mathrm{M}$ is the maximum regular graph relation having the set property, it's sufficient to prove that $\sim_\mathrm{M}$ is regular and has the set property, and in order to prove this, it's sufficient to prove that $\sim_\mathrm{M}$ is transitive, since all the other requirements are obvious. However, things are not that simple:

\begin{theorem}
    $\sim_\mathrm{M}$ is not transitive.
\end{theorem}

\begin{proof}
    Consider the following graph $J_1$:
    
\begin{figure}[htbp]
\[\begin{tikzcd}
	\bullet \\
	\\
	\bullet
	\arrow[from=1-1, to=1-1, loop, in=55, out=125, distance=10mm]
	\arrow[bend left, from=1-1, to=3-1]
	\arrow[bend left, from=3-1, to=1-1]
\end{tikzcd}\]
\caption{$J_1$}
\end{figure}

    Note that $G_\varnothing\sim_\mathrm{M} J_1$ (it's easy to check that $B=(G_\varnothing[G_\varnothing]\cup J_1[J_1])\times(G_\varnothing[G_\varnothing]\cup J_1[J_1])$ is regular and satisfies (i)(ii)), $Q_1\sim_\mathrm{M} J_1$ (similarly).
    However, obviously, $G_\varnothing\not\sim_\mathrm{M} Q_1$, so $\sim_\mathrm{M}$ is not transitive.
\end{proof}

Nevertheless, this motivates us to make the following definition:

\begin{definition}
    $G\sim_\mathrm{m} G'$ iff there exists some $B\subseteq (G[G]\cup G'[G'])\times(G[G]\cup G'[G'])$ such that $B$ is a regular graph relation satisfying (i$^+$)(ii) on $G[G]\cup G'[G']$ and $(G,G')\in B$, where (i$^+$) is: If non-*-graphs $G\not\cong G'$ and $G\sim G'$, then both of them are not $\sim$-extensional.
\end{definition}

Obviously, (i$^+$) implies (i).

\begin{theorem}
    $\sim_\mathrm{m}$ is the maximum regular graph relation satisfying (i$^+$)(ii). In particular, it has the set property.
\end{theorem}

\begin{proof}
    Again, it's sufficient to prove that $\sim_\mathrm{m}$ is transitive.

    If $G_1\sim_\mathrm{m} G_0$ (i.e., there exists some corresponding $B_1$) and $G_0\sim_\mathrm{m} G_2$ (i.e., there exists some corresponding $B_2$), we prove that $G_1\sim_\mathrm{m} G_2$. Replace $B_1$ (resp. $B_2$) with the transitive closure of $B_1 \cup \cong^*$ (resp. $B_2 \cup \cong^*$) and let $B$ be the transitive closure of $B_1\cup B_2$, obviously, its restriction to $(G_1[G_1]\cup G_2[G_2])\times (G_1[G_1]\cup G_2[G_2])$ is regular.
    Below we prove that its restriction to $(G_1[G_1]\cup G_2[G_2])\times (G_1[G_1]\cup G_2[G_2])$ satisfies (i$^+$)(ii).

    (i$^+$): If non-*-graphs $G\not\cong G'$ and $GBG'$, then there exists a sequence of non-*-graphs $G^i$ and $\delta_i\in\{1,2\}$ such that $G=G^0B_{\delta_0}G^1\dots G^{n-1}B_{\delta_{n-1}}G^n=G'$.
    Since $G\not\cong G'$, there exists some $i$ such that $G\cong G^i\not\cong G^{i+1}$, so $G^i$ is not $B_{\delta_i}$-extensional, and then $G$ is not $B$-extensional.
    Similarly, $G'$ is not $B$-extensional, so the requirement is satisfied.

    (ii): If $G$ has a point having no parents, $a\in G\backslash\{p_G\}$ such that $G\not\approx G[a]$ and $GBG[a]$, we proceed by cases.

    If $G$ is a $*$-graph, i.e. $G=G[b]^*$, then $b\neq a$ (since $G\not\approx G[a]$ but $G\approx G[b]$) $\in G\backslash\{p_G\}$ and $G[b]BG[a]$, so the requirement is satisfied.
    If $G[a]$ is a $*$-graph, similarly, the requirement is satisfied.
    Below we assume that neither of them is a $*$-graph.

    If $G\cong G[a]$, denote the isomorphism by $f$, we have that $a\neq f(a)\in G\backslash\{p_G\}$ and $G[a]BG[f(a)]$, so the requirement is satisfied.

    If $G\not\cong G[a]$, by the same argument used above to prove that $B$ satisfies (i$^+$), we have that $G$ is not $B_{\delta}$-extensional and $G[a]$ is not $B_{\delta'}$-extensional for some $\delta,\delta'\in\{1,2\}$.
    
    If $G$ is not $B_{\delta}$-extensional, then either there exist some $b\neq c\in G\backslash\{p_G\}$ such that $G[b]B_{\delta}G[c]$, so the requirement is satisfied; or there exists some $b\in G\backslash\{p_G\}$ such that $GB_{\delta}G[b]$, note that $G\not\approx G[b]$ (otherwise $G=G[b]^*$ is a $*$-graph), so by the fact that $B_{\delta}$ satisfies (ii), the requirement is satisfied.

    If $G[a]$ is not $B_{\delta'}$-extensional, then there exist some $b\neq c\in G\backslash\{p_G\}$ such that $G[b]B_{\delta'}G[c]$, so the requirement is satisfied.
\end{proof}

By the way, the argument used above to prove that $B$ satisfies (ii) can in fact be used to prove:

\begin{corollary}
    (i$^+$) implies (ii).
\end{corollary}

$\sim_\mathrm{m}$ has the following natural variants:

\begin{definition}
    $G\sim_{\mathrm{m}_Q}G'$ ($Q$ is one of the definable graphs) iff there exists some $B\subseteq (G[G]\cup G'[G'])\times(G[G]\cup G'[G'])$ such that $B$ is a regular graph relation satisfying (i$_Q^+$)(ii) on $G[G]\cup G'[G']$ and $(G,G')\in B$, where (i$_Q^+$) is: If non-*-graphs $G\not\cong G'$ and $G\sim G'$, then both of them are either $\cong Q$ or not $\sim$-extensional.
\end{definition}

Obviously, (i$_Q^+$) implies (i) and (i$^+$) implies (i$_Q^+$).

\begin{theorem}
    $\sim_{\mathrm{m}_Q}$ is the maximum regular graph relation satisfying (i$_Q^+$)(ii). In particular, it has the set property.
\end{theorem}
\begin{proof}
    The proof is similar to that of Theorem 7.
\end{proof}

\begin{corollary}
(finitary arithmetic)
    \begin{enumerate}[label=(\arabic*)]
        \item $\zfc^-+\afa^{\sim_\mathrm{m}}$ is consistent relative to $\zfc^-$.
        \item If $\zfc^-\vdash$``$\sim_{\mathrm{m}_Q}$ is absolute for $\uni_c^{\sim_{\mathrm{m}_Q}}$'', then $\zfc^-+\afa^{\sim_{\mathrm{m}_Q}}$ is consistent relative to $\zfc^-$.
    \end{enumerate}
\end{corollary}

\begin{theorem} (finitary arithmetic) For every $\sim$ such that $\zfc^-\vdash$``$\sim$ is a regular bisimulation'':
\begin{enumerate}[label=(\arabic*)]
        \item (If $\zfc^-$ is consistent, then) $\zfc^-\not\vdash\afa^{\sim_\mathrm{m}}\Leftrightarrow\afa^\sim$.
        \item If $\zfc^-\vdash$``$\sim_{\mathrm{m}_Q}$ is absolute for $\uni_c^{\sim_{\mathrm{m}_Q}}$'', then (if $\zfc^-$ is consistent, then) $\zfc^-\not\vdash\afa^{\sim_{\mathrm{m}_Q}}\Leftrightarrow\afa^\sim$.
    \end{enumerate}
\end{theorem}

\begin{proof}
    (1) Again, it's sufficient to prove that $\afa^{\sim_\mathrm{m}}\Rightarrow\neg\afa^\sim$ under $\ZFC^-$. Consider the following graph $J_2$:
    
\begin{figure}[htbp]
\[\begin{tikzcd}
	{c} & {a} && {b} \\
	\\
	&& {d}
	\arrow[from=1-1, to=1-1, loop, in=145, out=215, distance=10mm]
	\arrow[from=1-2, to=1-1]
	\arrow[bend left, from=1-2, to=1-4]
	\arrow[from=1-2, to=3-3]
	\arrow[bend left, from=1-4, to=1-2]
	\arrow[from=1-4, to=3-3]
\end{tikzcd}\]
\caption{$J_2$}
\end{figure}

    On one hand, $J_2$ is not $\sim_\mathrm{m}$-extensional, since $J_2[a]\sim_\mathrm{m} J_2[b]$: let $B=\{(J_2[v],J_2[v])\mid v\in J_2\}\cup \{(J_2[a],J_2[b]),(J_2[b],J_2[a])\}$, it's easy to check that it is regular and satisfies (i$^+$)(ii).
    
    On the other hand, $J_2$ is $\sim$-extensional:
    suppose not, since $J_2[d]\not\sim$ any other graph, either $J_2[c]\sim J_2[a]$ or $J_2[b]$, so $J_2[c]\sim J_2[d]$, a contradiction;
    or $J_2[a]\sim J_2[b]$, so $J_2[c]\sim J_2[a]$, which brings us back to
the former case.

    In conclusion, $J_2$ witnesses that $\afa^{\sim_\mathrm{m}}\Rightarrow\neg\afa^\sim$, thus ends the proof.

    (2) The proof is similar to that of (1).
\end{proof}

Theorem 9 gives us another anti-foundation axiom family, $\afa^{\sim_\mathrm{m}}$: a graph is an exact picture iff it is $\sim_\mathrm{m}$-extensional, and $\afa^{\sim_{\mathrm{m}_Q}}$ ($Q$ is one of the definable graphs): a graph is an exact picture iff it is $\sim_{\mathrm{m}_Q}$-extensional.

Using $\sim_{\mathrm{m}_Q}$, we can prove that the answer to Question 2 is negative.

\begin{theorem}
    There doesn't exist the maximum regular graph relation having the set property.
\end{theorem}
\begin{proof}
    The proof is similar to that of Theorem 6.
    Assume the contrary that there exists the maximum regular graph relation having the set property $\sim$, consider the graph $J_1$.

    Note that $G_\varnothing \sim_{\mathrm{m}_{G_\varnothing}} J_1$ (it's easy to check that $B=(G_\varnothing[G_\varnothing]\cup J_1[J_1])\times (G_\varnothing[G_\varnothing]\cup J_1[J_1])$ is regular and satisfies (i$_{G_\varnothing}^+$)(ii)), $Q_1\sim_{\mathrm{m}_{Q_1}} J_1$ (similarly), so $G_\varnothing\sim J_1$, $Q_1\sim J_1$, and then $G_\varnothing\sim Q_1$.
    However, since $\sim$ has the set property, $G_\varnothing\not\cong Q_1\Rightarrow G_\varnothing\not\sim Q_1$, a contradiction!
\end{proof}

However, although there doesn't exist the maximum regular graph relation having the set property, there does exist the minimum $\sim$-extensionality induced by a regular graph relation $\sim$ having the set property, and it is just $\sim_\mathrm{M}$-extensionality.

\begin{theorem}
    The following are mutually equivalent:
    \begin{enumerate}[label=(\Roman*)]
        \item $G$ is $\sim_\mathrm{M}$-extensional.
        \item $G$ is extensional and quasi-well-founded, that is, $G$ is well-founded, or $G$ contains one and only one $a$ such that $G[a]\cong Q_1$ and every infinite path in $G$ ends with $a \to a \to\dots$.
        \item $G$ is $\sim_\mathrm{M}'$-extensional, where $G_1 \sim_\mathrm{M}' G_2$ iff $G_1 \cong^* G_2$, or both $G_1$ and $G_2$ belong to one and the same of the following:
        \begin{enumerate}[label=(Class \arabic*)]
            \item $G$ is not extensional-and-quasi-well-founded and has a $G_\varnothing$ as a descendant subgraph, or $G$ is $G_\varnothing$, up to isomorphism.
            \item $G$ is not extensional-and-quasi-well-founded and has no $G_\varnothing$ as a descendant subgraph and has a $Q_1$ as a descendant subgraph, or $G$ is $Q_1$, or $G$ is $Q_1^*$, up to isomorphism.
            \item $G$ is not extensional-and-quasi-well-founded and has no $G_\varnothing$ as a descendant subgraph and has no $Q_1$ as a descendant subgraph, up to isomorphism.
        \end{enumerate}
    \end{enumerate}
    Moreover, $\sim_\mathrm{M}'$ is a regular graph relation having the set property. 
\end{theorem}

\begin{proof}
    First, we prove that $\sim_\mathrm{M}'$ is a regular graph relation having the set property. Obviously, $\sim_\mathrm{M}'$ is regular, so it's sufficient to prove that $\sim_\mathrm{M}'$ satisfies (i)(ii).

    (i): If non-*-graphs $G\not\cong G'$ and $G\sim_\mathrm{M}' G'$, we may assume that $G\not\cong^* G'$, since otherwise both $G$ and $G'$ are not $\cong^*$-extensional, and then the requirement is satisfied. 
    The only non-trivial case is where both $G$ and $G'$ belong to the class 3, and the only non-trivial subcase is where $G$ is not well-founded and has no $G_\varnothing$ or $Q_1$ as a descendant subgraph. Now, there exist at least two vertices $b,c \in G$, and both $G[b]$ and $G[c]$ belong to the class 3. Thus, $G[b] \sim_\mathrm{M}' G[c]$, so the requirement is satisfied.

    (ii): The proof is similar to that of (i).
    
    Second, we prove that (I) implies (III). It's sufficient to prove that $G\sim_\mathrm{M}' G'$ implies $G\sim_\mathrm{M} G'$. Since $\sim_\mathrm{M}'$ is regular and satisfies (i)(ii), this follows by the definition of $\sim_\mathrm{M}$.

    Third, we prove that (III) implies (II). This follows by the definition of $\sim_\mathrm{M}'$.

    Last, we prove that (II) implies (I). It's sufficient to prove that if $G$ is extensional and not $\sim_\mathrm{M}$-extensional, then it is not quasi-well-founded.
    
    Start from $G[b] \sim_\mathrm{M} G[c]$, i.e., there exists some $B$ such that $B$ satisfies (i)(ii) and $G[b]BG[c]$. If $G[b]\cong G[c]$, since by the Mostowski collapse lemma, being extensional and quasi-well-founded implies being $\cong$-extensional, we are done. Otherwise, either $G[b]$ or $G[c]$, say $G[b]$, is not $B$-extensional. If $G[b]$ has a point having a parent, then we are done. Otherwise, there exist some $d \not= e \in G[b]\backslash\{b\}$ such that $G[d]BG[e]$. Repeating this argument, we get an infinite path which doesn't end with $a \to a \to\dots$ for any $a$.
\end{proof}

\begin{corollary}
    $\sim_\mathrm{M}'$-extensionality is the minimum $\sim$-extensionality induced by a regular graph relation $\sim$ having the set property.
\end{corollary}

\printbibliography

@book{Aczel,
  author    = {Peter Aczel},
  title     = {Non-well-founded sets},
  publisher = {CSLI Publications},
  year      = {1988},
  address   = {Stanford},
}

@article{Daghighi,
      title={The foundation axiom and elementary self-embeddings of the universe}, 
      author={Ali Sadegh Daghighi and Mohammad Golshani and Joel David Hamkins and Emil Jeřábek},
      year={2014},
      eprint={1311.0814},
      archivePrefix={arXiv},
      url={https://arxiv.org/abs/1311.0814}, 
}

@article{Ju,
  title={Comparing anti-foundation axioms by comparing identity conditions for sets},
  author={Ju, Daheng and Jing, Hang Qi},
  journal={Philosophia Mathematica},
  volume={34},
  number={2},
  pages={294--307},
  year={2026},
  publisher={Oxford University Press}
}

@article{Rieger,
  title={A contribution to G{\"o}del's axiomatic set theory, I},
  author={Rieger, Ladislav},
  journal={Czechoslovak Mathematical Journal},
  volume={7},
  number={3},
  pages={323--357},
  year={1957},
  publisher={Institute of Mathematics, Academy of Sciences of the Czech Republic}
}

@article{Incurvati,
  title={The graph conception of set},
  author={Incurvati, Luca},
  journal={Journal of Philosophical Logic},
  volume={43},
  number={1},
  pages={181--208},
  year={2014},
  publisher={Springer}
}

\end{document}